\documentclass[11pt]{amsart}
\usepackage{amsmath}
\usepackage{amsfonts}
\usepackage{amsthm}
\usepackage{amssymb}
\usepackage{tikz-cd}
\usepackage{graphics}
\usepackage{array}
\usepackage{enumerate}
\usepackage{dsfont}
\usepackage{color}
\usepackage{wasysym}
\usepackage{hyperref}
\usepackage{pdfsync}

\newcommand{\bel}[1]{\begin{equation}\label{#1}}

\newcommand{\be}{\begin{equation}}

\newcommand{\ba}{\begin{eqnarray}}
\newcommand{\ea}{\end{eqnarray}}

\newcommand{\qe}{\end{equation}}

\newcommand{\Hmm}[1]{\leavevmode{\marginpar{\tiny%
$\hbox to 0mm{\hspace*{-0.5mm}$\leftarrow$\hss}%
\vcenter{\vrule depth 0.1mm height 0.1mm width \the\marginparwidth}%
\hbox to
0mm{\hss$\rightarrow$\hspace*{-0.5mm}}$\\\relax\raggedright #1}}}

\theoremstyle{theorem}
\newtheorem{thm}{Theorem}[section]
\theoremstyle{conjecture}
\newtheorem{conj}[thm]{Conjecture}
\newtheorem{prop}[thm]{Proposition}
\theoremstyle{example}

\theoremstyle{corollary}
\newtheorem{coro}[thm]{Corollary}
\theoremstyle{lemma}
\newtheorem{lemma}[thm]{Lemma}
\theoremstyle{definition}
\newtheorem{defi}[thm]{Definition}
\theoremstyle{proof}

\theoremstyle{remark}

\theoremstyle{Acknowledgements}
\newtheorem*{ac}{Acknowledgements}

\begin{document}

\title{Some non-vanishing results on log canonical pairs of dimension 4}
\author{Fanjun Meng}
\address{Department of Mathematics, Northwestern University, 2033 Sheridan Road, Evanston, IL 60208, USA}
\email{fanjunmeng2022@u.northwestern.edu}

\begin{abstract}
Let $(X,\Delta)$ be a log canonical pair over $\mathbb{C}$ with $X$ a normal projective variety, $\Delta$ an effective $\mathbb{Q}$-divisor, and $K_X+\Delta$ nef. We give a non-vanishing criterion for $K_X+\Delta$ in dimension $n$ with $X$ uniruled, assuming various conjectures of LMMP in dimensions (up to) $n-1$ or $n$, and a semi-ampleness criterion in the irregular case. In particular, we obtain that if $X$ is a uniruled $4$-fold, then $\kappa(K_X+\Delta)\geq 0$ and if $X$ is a $4$-fold with $q(X)>0$, then $K_X+\Delta$ is semi-ample.
\end{abstract}

\maketitle

\section{Introduction}
In birational geometry, a minimal model is expected to have the nice property that its canonical divisor is semi-ample. This is called the abundance conjecture. There is also a log version of the minimal model program, and the abundance conjecture can be generalized to log canonical pairs. All the varieties we consider below are projective varieties over $\mathbb{C}$.

\begin{conj}[Log abundance conjecture]
Let $(X,\Delta)$ be a log canonical pair. If $K_X+\Delta$ is nef, then $K_X+\Delta$ is semi-ample.
\end{conj}

There is another famous conjecture which is related to log abundance.

\begin{conj}[Good minimal model conjecture]
Let $(X,\Delta)$ be a log canonical pair. If $K_X+\Delta$ is pseudo-effective, then $(X,\Delta)$ has a good minimal model.
\end{conj}

The log abundance conjecture is only known to be true in dimensions up to $3$. In dimension $1$, it is trivial. In dimension $2$, it was proven by Kawamata in \cite{Kaw79} and Fujita in \cite{Fuj84}. In dimension $3$, it was proven by Keel, Matsuki and $\rm M^{c}$Kernan in \cite{KMM94}.

To prove the log abundance conjecture, the first natural step is to prove the following non-vanishing conjecture. It is implied by log abundance and the log minimal model program.

\begin{conj}[Non-vanishing conjecture]
Let $(X,\Delta)$ be a log canonical pair. If $K_X+\Delta$ is pseudo-effective, then $\kappa(K_X+\Delta)\geq0$.
\end{conj}

In this paper, we prove the following non-vanishing result for $(X,\Delta)$ when $X$ is uniruled. We use some reduction techniques from \cite{KMM94} to improve singularities. For the definition of the termination with scaling for pairs, we follow the definitions in \cite{Bir10}, cf. Definition \ref{LMMPS} too.

\begin{thm}[cf. Theorem \ref{Uni}]\label{Uni1}
Assume the existence of good minimal models for klt pairs in dimensions up to $n-1$, the log abundance conjecture in dimension $n-1$ and the termination with scaling for $\mathbb{Q}$-factorial dlt pairs in dimension $n$. Let $(X,\Delta)$ be a log canonical pair of dimension $n$ with $X$ uniruled. If $K_X+\Delta$ is nef, then $\kappa(K_X+\Delta)\geq0$.
\end{thm}

The first two assumptions are satisfied in dimensions up to 3 by an accumulation of many results. Birkar proved the termination with scaling for $\mathbb{Q}$-factorial dlt pairs in dimension $4$ in \cite[Lemma 3.8]{Bir10}. Recently, Moraga proved the termination of $(K_X+\Delta)$-flips for every log canonical pair $(X,\Delta)$, assuming that the $\mathbb{Q}$-Cartier $\mathbb{Q}$-divisor $K_X+\Delta$ is pseudo-effective and $X$ is of dimension $4$ in \cite[Theorem 1]{Mor18}. Therefore we deduce the following:

\begin{coro}[cf. Corollary \ref{Uni4}]\label{p}
Let $(X,\Delta)$ be a log canonical pair such that $X$ is a uniruled $4$-fold. If $K_X+\Delta$ is pseudo-effective, then $\kappa(K_X+\Delta)\geq0$.
\end{coro}

Moving on to the next result, the irregularity $q(X)$ of a projective variety is defined as the irregularity of any resolution of $X$. It is well defined and invariant under birational equivalence. A projective variety is called irregular if $q(X)>0$.

Using the same reduction techniques as in the proof of Theorem \ref{Uni1}, it is not hard to deduce from \cite[Theorem 1.3]{BC15} and \cite[Proposition 3.2]{LP18} (cf. Theorem \ref{AFS2} and Proposition \ref{AFS1}) the following log abundance statement for irregular varieties.

\begin{thm}[cf. Theorem \ref{Lq}]\label{qq}
Assume the existence of good minimal models for klt pairs in dimensions up to $n-1$, the log abundance conjecture in dimension $n-1$ and the termination with scaling for $\mathbb{Q}$-factorial dlt pairs in dimension $n$. Let $(X,\Delta)$ be a log canonical pair of dimension $n$ with $q(X)>0$. If $K_X+\Delta$ is nef, then $K_X+\Delta$ is semi-ample.
\end{thm}

As before, the assumptions in Theorem \ref{qq} are true when $n=4$. Thus it implies the log abundance conjecture for $(X,\Delta)$ where $X$ is of dimension $4$ and $q(X)>0$ unconditionally. Fujino has proven the abundance conjecture and the good minimal model conjecture for irregular canonical $4$-folds (with no pairs) in \cite[Theorem 1.3 and Theorem 1.4]{Fuj10}. Our result generalizes this to log canonical pairs, and our method of proof is quite different from his. Hu studied algebraic fibre spaces over irregular varieties with Albanese fibre of general type in \cite{Hu16} and got some abundant results for log canonical pairs, which is related to our theorem above.

\begin{coro}[cf. Corollary \ref{Lq4}]\label{q}
Let $(X,\Delta)$ be a log canonical pair such that $X$ is of dimension $4$ and $q(X)>0$. If $K_X+\Delta$ is nef, then $K_X+\Delta$ is semi-ample.
\end{coro}

By Moraga's result again, we know the termination of $(K_X+\Delta)$-flips for every log canonical pair $(X,\Delta)$, assuming that the $\mathbb{Q}$-Cartier $\mathbb{Q}$-divisor $K_X+\Delta$ is pseudo-effective and $X$ is of dimension $4$. We deduce the following:

\begin{coro}
Let $(X,\Delta)$ be a log canonical pair such that $X$ is of dimension $4$ and $q(X)>0$. If $K_X+\Delta$ is pseudo-effective, then $(X,\Delta)$ has a good minimal model.
\end{coro}

\begin{ac}
I would like to express my sincere gratitude to my advisor Mihnea Popa for telling me related problems, illuminating me when I meet difficulties and selflessly supporting me in every way. I would like to thank Mircea Musta$\c{t}$${\breve{a}}$, Charles Stibitz and Yuan Wang for answering my questions, helpful discussions and suggesting references.
\end{ac}

\section{preliminaries}

For definitions regarding singularities of pairs and the log minimal model program, we follow \cite{KM98}, except that we restrict to effective pairs. For properties about Albanese varieties, Albanese maps and Albanese morphisms for singular varieties, we refer to \cite{Lan83} and \cite{Wan16}.

We start by recalling the definition for abundant divisors.

\begin{defi}
Let $D$ be a nef $\mathbb{Q}$-Cartier $\mathbb{Q}$-divisor on a normal projective variety $X$. The numerical Kodaira dimension $\nu(D)$ of $D$ is defined to be $\rm max$\{$e;$ $D^e$ $\not\equiv0$\}. We say that $D$ is an abundant divisor if $\nu(D)=\kappa(D)$.
\end{defi}

Next we recall two results about the existence of dlt models and terminal models. The first one is due to Hacon; cf. \cite[Theorem 3.1]{KK10}.

\begin{thm}[Dlt blow-up, dlt model]\label{dlt}
Let $(X,\Delta)$ be a log canonical pair. We can find a new pair $(Y,\Delta_Y)$ and a birational morphism $f: Y\to X$ which satisfies the following:

1. $Y$ is $\mathbb{Q}$-factorial,

2. $(Y,\Delta_Y)$ is dlt,

3. $K_Y+\Delta_Y\sim_{\mathbb{Q}}f^*(K_X+\Delta)$.\\
We call $(Y,\Delta_Y)$ a dlt blow-up of $(X,\Delta)$.
\end{thm}

The next is due to Birkar, Cascini, Hacon and $\rm M^{c}$Kernan; cf. \cite[Corollary 1.4.3]{BCHM} and the statement right after that.

\begin{thm}[Terminalization, terminal model]\label{TerM}
Let $(X,\Delta)$ be a klt pair. We can find a new pair $(Y,\Delta_Y)$ and a birational morphism $f: Y\to X$ which satisfies the following:

1. $Y$ is $\mathbb{Q}$-factorial,

2. $(Y,\Delta_Y)$ is terminal,

3. $K_Y+\Delta_Y\sim_{\mathbb{Q}}f^*(K_X+\Delta)$.\\
We call $(Y,\Delta_Y)$ a terminalization of $(X,\Delta)$.
\end{thm}

We also recall a result about dlt adjunction.

\begin{defi}
Let $(X,\Delta)$ be a log canonical pair, $E$ a prime divisor over $X$, and $a(E, X, \Delta)$ the discrepancy coefficient of $E$ with respect to $(X,\Delta)$. A closed subvariety $W$ of $X$ is called a log canonical center or lc center if there exists a resolution $f: Y \to X$ and a divisor $E$ on $Y$ such that $a(E, X, \Delta)=-1$ and $f(E)=W$.
\end{defi}

The following result can be found in \cite[Section 4.2]{Kol13}.

\begin{thm}[Dlt adjunction]\label{dltad}
Let $(X,\Delta)$ be a dlt pair, and $Z$ a log canonical center for the pair. Then:

1. There exists an effective $\mathbb{Q}$-divisor $\Delta_Z$ on $Z$ such that $(Z, \Delta_Z)$ is dlt.

2. $(K_X+\Delta)|_{Z}\sim_{\mathbb{Q}}K_Z+\Delta_Z$.
\end{thm}

Next we include three results about the behavior of semi-ampleness and Kodaira dimension on algebraic fibre spaces.

\begin{defi}
A surjective morphism $f: X\to Y$ between two normal projective varieties is called an algebraic fibre space if $f_*\mathcal{O}_X=\mathcal{O}_Y$.
\end{defi}

The following result is due to Lazi$\mathrm{\acute{c}}$ and Peternell; cf. \cite[Lemma 2.7 and Proposition 3.2]{LP18}.

\begin{prop}\label{AFS1}
Assume the existence of good minimal models for klt pairs in dimensions up to $n-1$. Let $(X,\Delta)$ be a $\mathbb{Q}$-factorial projective klt pair such that $K_X+\Delta$ is nef. If there exists a fibration $X \to Z$ to a normal projective variety $Z$ such that $\rm dim$\,$Z\geq1$ and $K_X+\Delta$ is not big over $Z$, then $K_X+\Delta$ is semi-ample.
\end{prop}

The following one is due to Birkar and Chen; the first part is just \cite[Theorem 1.3]{BC15} and the second statement about semi-ampleness is a direct consequence of the existence of good minimal models due to \cite[Theorem 4.3]{GL13}. For the definition of a projective variety of maximal Albanese dimension, cf. Definition \ref{mad}.

\begin{thm}\label{AFS2}
Let $(X,\Delta)$ be a projective klt pair, $f: X \to Z$ an algebraic fibre space, and $Z$ a normal projective variety of maximal Albanese dimension. If $K_X+\Delta$ is big over $Z$, then $(X,\Delta)$ has a good minimal model and
\begin{center}
$\kappa(K_{X}+\Delta)\geq\kappa(K_{X_{z}}+\Delta|_{X_z})+\kappa(Z)$
\end{center}
where $X_{z}$ is a sufficiently general fibre of $f$.

If $K_{X}+\Delta$ is also nef and $X$ is $\mathbb{Q}$-factorial, then $K_{X}+\Delta$ is semi-ample.
\end{thm}

We next state a generalization to the log canonical case of Iitaka's conjecture when the base is of general type. This is due to Fujino \cite[Theorem 1.7]{Fuj17}.

\begin{thm}[Addition formula]\label{AFS3}
Let $(X,\Delta)$ be a log canonical pair, and $f: X \to Z$ an algebraic fibre space. Assume $Z$ is of general type. Then
\begin{center}
$\kappa(K_{X}+\Delta)=\kappa(K_{X_{z}}+\Delta|_{X_z})+\rm dim$\,$Z$
\end{center}
where $X_{z}$ is a sufficiently general fibre of $f$.
\end{thm}

We now discuss the log minimal model program with scaling. It will be used frequently in the later sections. We follow the introduction in \cite[Section 3]{Bir10} but we restrict to the absolute setting and to $\mathbb{Q}$-divisors.

Let $(X,\Delta+F)$ be a $\mathbb{Q}$-factorial log canonical pair such that $F$ is an effective $\mathbb{Q}$-divisor and $K_X+\Delta+F$ is nef. We can define a number $\lambda$
\begin{center}
$\lambda:=\rm inf$$\{s\geq0$ $|$ $K_X+\Delta+sF$ is nef\}
\end{center}
Notice that this infimum is a minimum by the definition of nef divisors.

We recall the following well-known fact and include a detailed proof.

\begin{lemma}\label{Ra}
In the setting above, $\lambda$ is a rational number. If $\lambda>0$, we can find a rational curve $C$ generating an extremal ray $R$ in the cone $\overline{NE}(X)$ such that $(K_X+\Delta+ \lambda F)\cdot C=0$ and $(K_X+\Delta)\cdot C<0$.
\end{lemma}
\begin{proof}
We only need to prove this when $\lambda>0$. By the cone theorem for log canonical pairs (see \cite{Fuj11}), we have
\begin{center}
$\overline{NE}(X)=\overline{NE}(X)_{K_X+\Delta\geq0}+\Sigma R_i$
\end{center}
Since $X$ is $\mathbb{Q}$-factorial, we can assume $K_X+\Delta$ and $F$ are integral Cartier divisors, because taking some multiples of $K_X+\Delta$ and $F$ in the definition of $\lambda$ will only affect $\lambda$ by a rational multiple. Since
\begin{center}
$K_X+\Delta+ \lambda F=\lambda(K_X+\Delta+F)+(1-\lambda)(K_X+\Delta)$
\end{center}
and $K_X+\Delta+F$ is nef, $\lambda$ is the same as the smallest number which makes $K_X+\Delta+ \lambda F$ nef on $\Sigma R_i$. From the cone theorem and the estimate of lengths of extremal rays (see \cite{Fuj11}), we know each extremal ray $R_i$ is generated by a rational curve $C_i$ such that $0<-(K_X+\Delta)\cdot C_i\leq2\,\rm dim$\,$X$. So $\lambda$ is the smallest number such that
\begin{center}
$(\lambda(K_X+\Delta+F)+(1-\lambda)(K_X+\Delta))\cdot C_i\geq0$, or
\end{center}

\begin{center}
$\frac{1-\lambda}{\lambda}\leq\frac{(K_X+\Delta+F)\cdot C_i}{-(K_X+\Delta)\cdot C_i}$
\end{center}
The number on the right of the second inequality has a rational minimum when it ranges over $C_i$, because the denominator is a positive integer controlled by $2\,\rm dim$\,$X$  and the numerator is a non-negative integer. Thus $\lambda$ reaches its minimum at some $C_i$ which means it is rational, $(K_X+\Delta+ \lambda F)\cdot C_i=0$ and $(K_X+\Delta)\cdot C_i<0$.
\end{proof}

\begin{defi}[LMMP with scaling]\label{LMMPS}
Let $(X,\Delta+F)$ be a $\mathbb{Q}$-factorial log canonical pair such that $F$ is an effective $\mathbb{Q}$-divisor and $K_X+\Delta+F$ is nef. Define
\begin{center}
$\lambda_0:=\rm inf$$\{\lambda\geq0$ $|$ $K_X+\Delta+ \lambda F$ is nef\}
\end{center}
If $\lambda_0=0$, we do nothing. If $\lambda_0>0$, from Lemma \ref{Ra} we can find a rational curve $C$ generating an extremal ray $R$ such that $(K_X+\Delta+ \lambda_0 F)\cdot C=0$ and $(K_X+\Delta)\cdot C<0$. By the log minimal model program, we can contract this $R$. If we get a Mori fibre space, we stop. Otherwise, it gives a divisorial contraction $X\to X_1$ or a log flip $X \dashrightarrow X_1$. We consider the new pair $(X_1, \Delta_1+ \lambda_0 F_1)$. $\Delta_1+ \lambda_0 F_1$ is the birational transform of $\Delta+ \lambda_0 F$. This new pair is still log canonical and $X_1$ is $\mathbb{Q}$-factorial from the log minimal model program and $(K_X+\Delta+ \lambda_0 F)\cdot C=0$. We can continue this process again by defining
\begin{center}
$\lambda_1:=\rm inf$$\{\lambda\geq0$ $|$ $K_{X_1}+\Delta_1+ \lambda F_1$ is nef\}
\end{center}
If $\lambda_1=0$, we do nothing. Otherwise we can find a rational curve $C_1$ generating an extremal ray $R_1$ such that $(K_{X_1}+\Delta_1+ \lambda_1 F_1)\cdot C_1=0$ and $(K_{X_1}+\Delta_1)\cdot C_1<0$. We can contract $R_1$ same as before. By continuing this process, we get a sequence $\lambda_0\geq \lambda_1\geq...\geq\lambda_n\geq...\geq0$. We obtain a special kind of log minimal model program which is called LMMP on $K_X+\Delta$ with scaling of $F$ or just ($K_X+\Delta$)-LMMP with scaling of $F$. When we say \emph{termination with scaling} we refer to the termination of such a LMMP.
\end{defi}
\begin{thm}\label{TerS}
Termination with scaling holds for every $\mathbb{Q}$-factorial dlt pair $(X,\Delta+F)$ in dimension $4$.
\end{thm}
\begin{proof}
See \cite[Lemma 3.8]{Bir10}.
\end{proof}

We also recall the recent result by Moraga \cite[Theorem 1]{Mor18} about the termination of $4$-fold flips for pseudo-effective pairs.

\begin{thm}\label{Ter}
Let $(X,\Delta)$ be a log canonical $4$-fold over an algebraically closed field of characteristic zero. Assume that the $\mathbb{Q}$-Cartier $\mathbb{Q}$-divisor $K_X+\Delta$ is pseudo-effective. Then any sequence of $(K_X+\Delta)$-flips terminates.
\end{thm}

Next we define the irregularity of a projective variety.

\begin{defi}
Let $X$ be a smooth projective variety. The irregularity $q(X)$ of $X$ is defined as $h^1(X,\mathcal{O}_X)$. If $X$ is a projective variety, irregularity $q(X)$ is defined as the irregularity of any resolution of $X$.
\end{defi}

It is easy to see the notion irregularity is well defined.

\begin{lemma}
If $X$ has rational singularities, then $q(X)=h^1(X,\mathcal{O}_X)$.
\end{lemma}

\begin{proof}
It comes from the definition of rational singularity and Leray spectral sequence.
\end{proof}

\begin{defi}
Let $X$ be a smooth projective variety, $\rm Alb$$(X)$ its Albanese variety, and $\alpha: X\to$ $\rm Alb$$(X)$ the Albanese morphism. We say that $X$ is of maximal Albanese dimension if $\rm dim$\,$\alpha(X)=\rm dim$\,$X$.
\end{defi}

\begin{defi}\label{mad}
Let $X$ be a projective variety. We say that $X$ is of maximal Albanese dimension if there is a resolution $\bar{X}$ of $X$ such that $\bar{X}$ is of maximal Albanese dimension.
\end{defi}

This notion is well defined because if two smooth projective varieties are birational to each other, they have isomorphic Albanese varieties and their Albanese maps commute with the birational map. It comes from that Albanese map and Albanese morphism coincide with each other for normal projective varieties with rational singularities. For details, see \cite{Lan83} and \cite[Section 2]{Wan16}.

\section{Non-vanishing for uniruled 4-folds  }

First we show that in order to prove non-vanishing or log abundance, we can reduce the questions to klt singularities under some assumptions. The following four lemmas follow from \cite[Section 7]{KMM94}. Some of them have been given detailed proofs in \cite{KMM94} while others only sketches. For completeness, we include some detailed proofs here.

\begin{lemma}\label{1}
Assume the termination with scaling for $\mathbb{Q}$-factorial dlt pairs in dimension $n$. Let $(X,\Delta+F)$ be a log canonical pair of dimension $n$ such that $K_X+\Delta+F$ is nef and $F=\llcorner \Delta+F \lrcorner$. Then we can find a new dlt pair $(Y,\Delta_Y+F_Y)$ such that

1. $K_{Y}+\Delta_Y+F_Y$ is nef, $F_Y=\llcorner \Delta_Y+F_Y \lrcorner$ and $Y$ is $\mathbb{Q}$-factorial and of dimension n,

2. $K_X+\Delta+F$ is semi-ample if and only if $K_{Y}+\Delta_Y+F_Y$ is semi-ample,

3. $\kappa(K_X+\Delta+F)\geq0$ if and only if $\kappa(K_{Y}+\Delta_Y+F_Y)\geq0$,

4. Either $Y$ is a Mori fibre space over a normal projective variety $Z$ such that $K_{Y}+\Delta_Y+F_Y$ is the pullback of a $\mathbb{Q}$-Cartier $\mathbb{Q}$-divisor from $Z$ and some irreducible component of $F_Y$ dominates $Z$, or $K_{Y}+\Delta_Y+(1-\epsilon)F_Y$ is nef for every non-negative small enough real number $\epsilon$.
\end{lemma}

\begin{proof}
We do a dlt blow-up by Theorem \ref{dlt} and we can assume our $(X,\Delta+F)$ is dlt and $X$ is $\mathbb{Q}$-factorial since nefness, semi-ampleness and Kodaira dimension do not change by surjective pullback. Then we run a $(K_X+\Delta)$-LMMP with scaling of $F$ and we know it terminates by assumption. We get a sequence of $\lambda_i$. If one of the $\lambda_i$ is strictly smaller than 1, then we pick the $\lambda_k<1$ such that $k$ is the smallest index. If $k=0$, we are done. Otherwise $\lambda_0=...=\lambda_{k-1}=1$ and all the contractions satisfy that

\begin{center}
$(K_{X_{i}}+\Delta_{i}+F_{i})\cdot C_{i}=0$ for $0\leq i<k$.
\end{center}
$X=X_0$, $\Delta=\Delta_0$ and $F=F_0$. We claim this equality implies that semi-ampleness and Kodaira dimension of $K_{X_{i}}+\Delta_{i}+F_{i}$ do not change during this LMMP with scaling. If we have a divisorial contraction $f: X_i \to X_{i+1}$, then
\begin{center}
$K_{X_i}+\Delta_i+F_{i}\sim_{\mathbb{Q}}f^*(K_{X_{i+1}}+\Delta_{i+1}+F_{i+1})+aE$
\end{center}
for some $a$ by the cone theorem. Here $E$ is the exceptional divisor of $f$. Since $E\cdot C_i<0$ and $(K_{X_{i}}+\Delta_{i}+F_{i})\cdot C_{i}=0$, we have $a=0$ and
\begin{center}
$K_{X_i}+\Delta_i+F_{i}\sim_{\mathbb{Q}}f^*(K_{X_{i+1}}+\Delta_{i+1}+F_{i+1})$.
\end{center}
We know semi-ampleness, Kodaira dimension and nefness do not change by surjective pullback. Our claim has been proven. If $f$ is a log flip, the proof is almost same. We omit it. By our discussion so far, we know that if such a $\lambda_k<1$ exists, then we have proven the lemma and it is in the first case of property $4$. If $\lambda_i=1$ for every $i$, then by the termination of LMMP with scaling, we know this program terminates at a Mori fibre space $g: X_i\to Z$ for some index $i$. By $(K_{X_{i}}+\Delta_{i}+F_{i})\cdot C_{i}=0$, we know $K_{X_{i}}+\Delta_{i}+F_{i}$ is a pullback from $Z$. We only need to prove that some irreducible component of $F_i$ dominates $Z$. Assume not, $K_{X_{i}}+\Delta_{i}+F_{i}$ coincides with $K_{X_{i}}+\Delta_{i}$ after restricting to a general fibre of $g$. This is absurd since $K_{X_{i}}+\Delta_{i}+F_{i}$ is nef and $K_{X_{i}}+\Delta_{i}$ is relatively anti-ample over $Z$.
\end{proof}

\begin{lemma}\label{2}
Assume the log abundance conjecture in dimension $n-1$. Let $(X,\Delta+F)$ be a $\mathbb{Q}$-factorial dlt pair of dimension $n$, and $f: X\to Z$ a Mori fibre space. Assume that $K_X+\Delta+F$ is nef, $F=\llcorner \Delta+F \lrcorner$ and $(X,\Delta+F)$ is the pullback of a $\mathbb{Q}$-Cartier $\mathbb{Q}$-divisor $D$ over $Z$. If an irreducible component $S$ of $F$ dominates $Z$, then $K_X+\Delta+F$ is semi-ample.
\end{lemma}

\begin{proof}
Since $S$ is an irreducible component of $F$, $S$ is a log canonical center of $(X,\Delta+F)$. By Theorem \ref{dltad}, we know there exists $\Delta_S$ such that $(S, \Delta_S)$ is dlt and $\rm dim$\,$S=n-1$. From $(K_X+\Delta+F)|_{S}\sim_{\mathbb{Q}}K_S+\Delta_S$, we have that $K_S+\Delta_S$ is nef on $S$. By assumption, we have that $K_S+\Delta_S$ is semi-ample. If we restrict $f$ to $S$, $f|_{S}$ is still surjective by assumption and $K_S+\Delta_S\sim_{\mathbb{Q}}(f|_{S})^*D$. It implies $D$ is semi-ample and then $K_X+\Delta+F$ is semi-ample since $K_X+\Delta+F$ is the pullback of $D$.
\end{proof}

Next lemma is \cite[Lemma 7.3]{KMM94}. It is an instance of Koll$\mathrm{\acute{a}}$r's injectivity theorem.

\begin{lemma}
Let $(X,\Delta)$ be a klt pair of dimension $n$, and $L$ an integral $\mathbb{Q}$-Cartier divisor such that $L-(K_X+\Delta)$ is semi-ample. Suppose $D$ and $D'$ are effective integral $\mathbb{Q}$-Cartier divisors such that $D+D'\in |m(L-(K_X+\Delta))|$ for some $m$. Then the homomorphisms induced by multiplication by $D$,
\begin{center}
$\phi^i_D : H^i(X,\mathcal{O}_X(L)) \to H^i(X,\mathcal{O}_X(L+D))$
\end{center}
are all injective.
\end{lemma}

\begin{lemma}\label{3}
Let $(X,\Delta+F)$ be a $\mathbb{Q}$-factorial dlt pair of dimension $n$, and $F=\llcorner \Delta+F \lrcorner$. Assume that $K_{X}+\Delta+(1-\epsilon)F$ is nef for every non-negative small enough rational number $\epsilon$. We have the following:

1. Assuming the non-vanishing conjecture for klt pairs in dimension $n$, then $\kappa(K_X+\Delta+F)\geq0$.

2. Assuming the log abundance conjecture in dimension $n-1$ and the log abundance conjecture for klt pairs in dimension $n$, then $K_X+\Delta+F$ is semi-ample.
\end{lemma}

\begin{proof}
The first statement is obvious. The proof of the second statement is the same as the proof of \cite[Corollary 7.4]{KMM94} because of our assumptions and the equivalence between the semi-log abundance conjecture and the log abundance conjecture in every dimension due to Fujino and Gongyo \cite[Theorem 4.3]{FG14}.
\end{proof}

We have finished the setup we need, and next we will prove the main result.

\begin{thm}\label{Uni}
Assume the existence of good minimal models for klt pairs in dimensions up to $n-1$, the log abundance conjecture in dimension $n-1$ and the termination with scaling for $\mathbb{Q}$-factorial dlt pairs in dimension $n$. Let $(X,\Delta)$ be a log canonical pair of dimension $n$ with $X$ uniruled. If $K_X+\Delta$ is nef, then $\kappa(K_X+\Delta)\geq0$.
\end{thm}

\begin{proof}
The first step is to take a dlt blow-up of our pair using Theorem \ref{dlt}. Then we have a new pair $(Y,\Delta_Y)$ such that $Y$ is $\mathbb{Q}$-factorial, $(Y,\Delta_Y)$ is dlt and $K_Y+\Delta_Y\sim_{\mathbb{Q}}f^*(K_X+\Delta)$. The last statement implies
\begin{center}
$\kappa(K_Y+\Delta_Y)=\kappa(K_X+\Delta)$
\end{center}
and $K_Y+\Delta_Y$ is nef. So we can assume $X$ is $\mathbb{Q}$-factorial and $(X,\Delta)$ is dlt without changing our problem.

Next we claim that we can reduce to terminal singularities. By Lemma \ref{1}, Lemma \ref{2} and Lemma \ref{3} combined with the fact that uniruledness is a birational invariant, we can assume $(X,\Delta)$ is klt. Then we take a terminalization of $(X,\Delta)$ using Theorem \ref{TerM}.

In summary, to prove our theorem, we only need to prove $\kappa(K_X+\Delta)\geq0$ when $(X,\Delta)$ is a terminal pair such that $X$ is uniruled and $\mathbb{Q}$-factorial with $K_X+\Delta$ nef. Since $X$ is uniruled and terminal, $K_X$ is not pseudo-effective. In particular, $K_X$ is not nef. Then we can run a $K_X$-MMP with scaling of $\Delta$ by assumption. We claim it must terminate at a Mori fibre space. If not, it terminates at a minimal model $Y$ such that $K_Y$ is nef. $Y$ is uniruled and terminal since we are running $K_X$-MMP with scaling. Thus $K_Y$ is not pseudo-effective, a contradiction.

So our $K_X$-MMP with scaling of $\Delta$ terminates at a Mori fibre space $f: Y \to Z$. We have that $Y$ is $\mathbb{Q}$-factorial and terminal. Let $\Delta_Y$ be the strict transform of $\Delta$. There exists a $0<\lambda\leq1$ such that $K_Y+\lambda\Delta_Y$ is nef, a pullback from $Z$,
\begin{center}
$\kappa(K_X+\Delta)\geq\kappa(K_Y+\lambda\Delta_Y)$ ($\kappa(K_{X_i}+\lambda_i\Delta_i)\geq\kappa(K_{X_{i+1}}+\lambda_{i+1}\Delta_{i+1})$)
\end{center}
and $(Y, \lambda\Delta_Y)$ is klt by the process of the MMP with scaling. If $\rm dim$\,$Z=0$, then
\begin{center}
$\kappa(K_X+\Delta)\geq\kappa(K_Y+\lambda\Delta_Y)=0$
\end{center}
and we are done. We assume $\rm dim$\,$Z>0$. In this case, $K_Y+\lambda\Delta_Y$ is not big over $Z$ since every fibre has positive dimension and $K_Y+\lambda\Delta_Y$ is a pullback from $Z$ which means it is trivial on the fibre. Combining with our hypothesis, all the assumptions in Proposition \ref{AFS1} are satisfied. We conclude $K_Y+\lambda\Delta_Y$ is semi-ample and thus $\kappa(K_Y+\lambda\Delta_Y)\geq0$. It implies $\kappa(K_X+\Delta)\geq0$.
\end{proof}

\begin{coro}\label{Uni4}
Let $(X,\Delta)$ be a log canonical pair such that $X$ is a uniruled $4$-fold. If $K_X+\Delta$ is pseudo-effective, then $\kappa(K_X+\Delta)\geq0$.
\end{coro}

\begin{proof}
We do a reduction step, which is likely well known. For completeness, we give a detailed proof here. We claim we can assume $K_X+\Delta$ is nef. If it is true, our corollary follows from Theorem \ref{Uni} since we know the assumptions in Theorem \ref{Uni} are true when $n=4$ by Theorem \ref{TerS}, the existence of good minimal models for klt pairs in dimensions up to $3$ and the log abundance conjecture in dimension $3$. Since $K_X+\Delta$ is pseudo-effective, we can run a log minimal model program for $K_X+\Delta$, and it will terminate by Theorem \ref{Ter}. Let $f: X_i \dashrightarrow X_{i+1}$ be a divisorial contraction or a log flip in the LMMP process and $\Delta_i$ be the strict transform of $\Delta$ on $X_i$. From the LMMP, we know that in each step
\begin{center}
$\kappa(K_{X_i}+\Delta_i)=\kappa(K_{X_{i+1}}+\Delta_{i+1})$.
\end{center}
If this LMMP terminates at a minimal model $(Y,\Delta_Y)$ of $(X,\Delta)$, then $K_Y+\Delta_Y$ is nef and we can reduce to the case when $K_X+\Delta$ is nef. If this LMMP terminates at a Mori fibre space $f: (Y,\Delta_Y)\to Z$, then $K_Y+\Delta_Y$ is relative anti-ample over $Z$. However we claim pseudo-effectivity is preserved during the LMMP. If it is true, $K_Y+\Delta_Y$ is pseudo-effective and it is pseudo-effective on the general fibre of $f$ by a similar method in the proof of our claim below. We have that $K_Y+\Delta_Y$ is both pseudo-effective and anti-ample on a general fibre of dimension bigger than $0$ and this is a contradiction. So our LMMP will terminate at a minimal model and we can do the reduction step stated at the beginning.

Thus we only need to prove that pseudo-effectivity is preserved during the LMMP now. If we have a divisorial contraction $f: X_i \to X_{i+1}$, then
\begin{center}
$K_{X_i}+\Delta_i\sim_{\mathbb{Q}}f^*(K_{X_{i+1}}+\Delta_{i+1})+aE$
\end{center}
where $a>0$ since the extremal ray being contracted is $K_{X_i}+\Delta_i$ negative. Here $E$ is the exceptional divisor of $f$. Choose an ample divisor $D$ on $X_{i+1}$, then $f^*(\frac{D}{k})$ is big on $X_i$ for any positive integer $k$ since $f$ is birational. We deduce that $K_{X_i}+\Delta_i+f^*(\frac{D}{k})$ is big since $K_{X_i}+\Delta_i$ is pseudo-effective. We know $aE$ is effective and $f$-exceptional. We also have
\begin{center}
$f^*(K_{X_{i+1}}+\Delta_{i+1}+\frac{D}{k})+aE\sim_{\mathbb{Q}}K_{X_i}+\Delta_i+f^*(\frac{D}{k})$.
\end{center}
Thus we have
\begin{center}
$\rm dim$\,$X=\kappa(K_{X_i}+\Delta_i+f^*(\frac{D}{k}))=\kappa(f^*(K_{X_{i+1}}+\Delta_{i+1}+\frac{D}{k})+aE)=\kappa(f^*(K_{X_{i+1}}+\Delta_{i+1}+\frac{D}{k}))=\kappa(K_{X_{i+1}}+\Delta_{i+1}+\frac{D}{k})$
\end{center}
This implies that $K_{X_{i+1}}+\Delta_{i+1}+\frac{D}{k}$ is big, so $K_{X_{i+1}}+\Delta_{i+1}$ is a limit of big divisors, hence it is pseudo-effective. If we have a log flip $f: X_i \dashrightarrow X_{i+1}$, the proof is similar using the fact that sections do not change by removing a closed set of codimension at least 2 from a normal variety.
\end{proof}

Analyzing the proof of Theorem \ref{Uni} and Corollary \ref{Uni4}, uniruledness is only used to rule out the case when we get a minimal model of $(X,\Delta)$. So we get the following:

\begin{coro}
To prove the non-vanishing conjecture for log canonical pairs of dimension $4$, we only need to prove the non-vanishing conjecture for every terminal variety $X$ such that $K_X$ is nef and $X$ is of dimension $4$.
\end{coro}

This corollary has already been obtained in dimension $4$, and in arbitrary dimensions with extra assumptions, in \cite[Section 8]{DHP}.

\section{Log abundance for 4-folds with positive irregularity}

In this section we prove Theorem \ref{qq}. Before we do this, let us state a structure theorem which follows almost verbatim from Kawamata's proof for \cite[Theorem 1]{Kaw81} about the algebraic fibre space structure for the Albanese morphism.

\begin{lemma}\label{Ab}
Let $X$ be a normal projective variety. Then one of the following is true:

1. Its Albanese morphism is an algebraic fibre space.

2. It has an $\acute{e}$tale cover $\bar{X}$ such that there exists an algebraic fibre space structure $\bar{X} \to W$ where $\rm dim$\,$W>0$ and $W$ is a normal projective variety of general type.
\end{lemma}

We next prove an easy lemma which will be used in the following.

\begin{lemma}\label{EA}
Let $(X,\Delta)$ be a log canonical pair, and $f: X \to Z$ an algebraic fibre space. Assume $K_X+\Delta$ is nef and $Z$ is of general type. If for a sufficiently general fibre $X_{z}$, we have $\kappa(K_{X_{z}}+\Delta|_{X_z})=\nu(K_{X_{z}}+\Delta|_{X_z})$, then $\kappa(K_{X}+\Delta)=\nu(K_{X}+\Delta)$.
\end{lemma}

\begin{proof}
We have an easy addition formula for numerical Kodaira dimension for nef divisors.
\begin{center}
$\nu(K_{X}+\Delta)\leq\nu(K_{X_{z}}+\Delta|_{X_z})+\rm dim$\,$Z$
\end{center}
The proof of it can be found in \cite[Lemma 2.3]{Nak04} when $X$ is smooth. The proof is verbatim when $X$ is normal. By assumption we have $\kappa(K_{X_{z}}+\Delta|_{X_z})=\nu(K_{X_{z}}+\Delta|_{X_z})$ for a sufficiently general fibre $X_{z}$. By Theorem \ref{AFS3}, we have $\kappa(K_{X}+\Delta)=\kappa(K_{X_{z}}+\Delta|_{X_z})+\rm dim$\,$Z$. Thus
\begin{center}
$\nu(K_{X}+\Delta)\leq\kappa(K_{X_{z}}+\Delta|_{X_z})+\rm dim$\,$Z=\kappa(K_{X}+\Delta)\leq\nu(K_{X}+\Delta)$.
\end{center}
We have $\kappa(K_{X}+\Delta)=\nu(K_{X}+\Delta)$.
\end{proof}

\begin{thm}\label{Lq}
Assume the existence of good minimal models for klt pairs in dimensions up to $n-1$, the log abundance conjecture in dimension $n-1$ and the termination with scaling for $\mathbb{Q}$-factorial dlt pairs in dimension $n$. Let $(X,\Delta)$ be a log canonical pair of dimension $n$ with $q(X)>0$. If $K_X+\Delta$ is nef, then $K_X+\Delta$ is semi-ample.
\end{thm}

\begin{proof}
As in the proof of Theorem \ref{Uni}, we can assume $(X,\Delta)$ is terminal and $X$ is $\mathbb{Q}$-factorial.

We use Lemma \ref{Ab}. Let us assume the case $1$ of the lemma is true, so the Albanese morphism of $X$ is an algebraic fibre space. If $K_X+\Delta$ is relatively big over $\rm Alb$$(X)$, then we use Theorem \ref{AFS2}. We deduce that $(X,\Delta)$ has a good minimal model and thus $K_X+\Delta$ is semi-ample. If $K_X+\Delta$ is not relatively big over $\rm Alb$$(X)$, then we use Proposition \ref{AFS1}. Since $q(X)>0$ and $X$ has rational singularities, we have $\rm dim$\,$\rm Alb$$(X)\geq1$. Thus Proposition \ref{AFS1} implies that $K_X+\Delta$ is semi-ample.

Assume now $X$ has an $\rm \acute{e}$tale cover $f: \bar{X} \to X$ such that there exists an algebraic fibre space structure $g: \bar{X} \to W$ with $\rm dim$\,$W>0$ and $W$ of general type. We pull back everything from $X$ to $\bar{X}$ and get a new pair $(\bar{X}, \bar{\Delta})$ such that $K_{\bar{X}}+\bar{\Delta}\sim_{\mathbb{Q}}f^*(K_X+\Delta)$. Here $\bar{X}$ is still a normal projective variety and $(\bar{X}, \bar{\Delta})$ is klt by \cite[Proposition 5.20]{KM98}. If $\bar{X}_{w}$ is a sufficiently general fibre of $g$, then $\rm dim$\,$\bar{X}_{w}\leq n-1$ since $\rm dim$\,$W>0$. For a sufficiently general $\bar{X}_{w}$, the pair $(\bar{X}_{w}, \bar{\Delta}|_{\bar{X}_w})$ is klt and $K_{\bar{X}_{w}}+\bar{\Delta}|_{\bar{X}_w}$ is nef. The existence of good minimal models for klt pairs in dimensions up to $n-1$ implies the log abundance conjecture for klt pairs in dimensions up to $n-1$ by \cite[Theorem 4.3]{GL13} and the existence of terminal models. We know that $K_{\bar{X}_{w}}+\bar{\Delta}|_{\bar{X}_w}$ is semi-ample and especially $\kappa(K_{\bar{X}_w}+\bar\Delta|_{\bar{X}_w})=\nu(K_{\bar{X}_w}+\bar\Delta|_{\bar{X}_w})$. By Lemma \ref{EA}, we have that $\kappa(K_{\bar{X}}+\bar{\Delta})=\nu(K_{\bar{X}}+\bar{\Delta})$ and $K_{\bar{X}}+\bar{\Delta}$ is semi-ample. Thus $K_X+\Delta$ is semi-ample.
\end{proof}

Since the assumptions in Theorem \ref{Lq} are true when $n=4$ as before, we have the following:

\begin{coro}\label{Lq4}
Let $(X,\Delta)$ be a log canonical pair such that $X$ is of dimension $4$ and $q(X)>0$. If $K_X+\Delta$ is nef, then $K_X+\Delta$ is semi-ample.
\end{coro}

\begin{coro}
Let $(X,\Delta)$ be a log canonical pair such that $X$ is of dimension $4$ and $q(X)>0$. If $K_X+\Delta$ is pseudo-effective, then $(X,\Delta)$ has a good minimal model.
\end{coro}

\begin{proof}
By Theorem \ref{Ter}, we can run a LMMP for $(X,\Delta)$ and it will terminate at a minimal model as in the proof of Corollary \ref{Uni4}. This minimal model still has positive irregularity and we can apply Corollary \ref{Lq4}.
\end{proof}
	
	\bibliographystyle{amsalpha}
	\bibliography{biblio}

\end{document}